\documentclass[12pt]{article}

\usepackage[centertags]{amsmath}
\usepackage{amsfonts}
\usepackage{amssymb}
\usepackage{amsthm}
\usepackage{newlfont}

\newlength{\defbaselineskip}
\newcommand{\setlinespacing}[1]%
           {\setlength{\baselineskip}{#1 \defbaselineskip}}

\theoremstyle{plain}
\newtheorem{thm}{Theorem}[section]
\newtheorem{cor}[thm]{Corollary}

\newtheorem{exm}[thm]{Example}
\theoremstyle{definition}
\newtheorem{defn}[thm]{Definition}
\newtheorem{rem}[thm]{Remark}

\numberwithin{equation}{section}

\begin{document}

\newcommand{\ol }{\overline}
\newcommand{\ul }{\underline }
\newcommand{\ra }{\rightarrow }
\newcommand{\lra }{\longrightarrow }
\newcommand{\ga }{\gamma }
\newcommand{\st }{\stackrel }
\newcommand{\scr }{\scriptsize }

\title{\Large\textbf{On Varietal Capability of Direct Products of Groups and Pair of Groups}}
\author{\textbf{ Hanieh Mirebrahimi\footnote{Correspondence: hanieh${\_}$hmp@yahoo.com} and Behrooz Mashayekhy   
}\\
Department of Pure Mathematics,\\ Center of Excellence in Analysis on Algebraic Structures,\\
Ferdowsi University of Mashhad,\\
P. O. Box 1159-91775, Mashhad, Iran.}
\date{ }
\maketitle

\begin{abstract}
In this paper we give some conditions in which a direct product
of groups is $\mathcal{V}-$capable if and only if each of its
factors is $\mathcal{V}-$capable for some varieties $\mathcal{V}$. Moreover, we give some conditions in
which a direct product of a finite family of pairs of groups is
capable if and only if each of its factors is a capable pair of
groups.
\end{abstract}

\newpage

\section{Introduction and Motivation}
R. Bear $[1]$ initiated an investigation of the question which
conditions a group $G$ must fulfill in order to be the group of
inner automorphisms of a group $E$ that is ($G\cong E/Z(E)$).
Following M. Hall and J. K. Senior $[6]$ such a group $G$ is
called capable. Bear $[1]$ determined all capable groups which are
direct sums of cyclic groups. As P. Hall $[5]$ mentioned,
characterizations of capable groups are important in classifying
groups of prime-power order.

F. R. Beyl, U. Felgner and P. Schmid $[2]$ proved that every group
$G$ possesses a uniquely determined central subgroup $Z^*(G)$
which is minimal subject to being the image in $G$ of the center
of some central extension of $G$. This $Z^*(G)$ is characteristic
in $G$ and is the image of the center of every stem cover of $G$.
Moreover, $Z^*(G)$ is the smallest central subgroup of G  whose
factor group is capable $[2, Corollary\ 2.2]$. Hence $G$ is capable if and only if
$Z^*(G)=1$ $[2, corollary\ 2.3]$. They showed that the class of all capable groups is
closed under the direct products $[2, Proposition\ 6.1]$. Also, they presented a condition
in which the capablity of a direct product of finitely many of
groups implies the capablity of each of the factors $[2, Proposition\ 6.2]$. Moreover,
they proved that if $N$ is a central subgroup of $G$, then
$N\subseteq Z^*(G)$ if and only if the mapping $M(G)\rightarrow
M(G/N)$ is monomorphic $[2, Theorem\ 4.2]$.

Then M. R. Moghadam and S. Kayvanfar $[11]$ generalized the
concept of capability to $\mathcal{V}-$capability for a group
$G$. They introduced the subgroup $(V^*)^*(G)$ which is
associated with the variety $\mathcal{V}$ defined by the set of
laws $V$ and a group $G$ in order to establish a necessary and
sufficient condition under which $G$ can be
$\mathcal{V}-$capable $[11, Corollary\ 2.4]$. They also showed that the class of all
$\mathcal{V}-$capable groups is closed under the direct products $[11, Theorem\ 2.6]$.
Moreover, they exhibited a close relationship between the groups
$\mathcal{V}M(G)$ and $\mathcal{V}M(G/N)$, where $N$ is a normal
subgroup contained in the marginal subgroup of $G$ with respect
to the variety $\mathcal{V}$. Using this relationship, they gave
a necessary and sufficient condition for a group $G$ to be
$\mathcal{V}-$capable $[11, Theorem\ 4.4]$.

In this paper, in section 3, we present some conditions in which
the $\mathcal{V}-$capablity of a direct product of a finitely many
groups implies the $\mathcal{V}-$capablity of each of its factors.

In continue, we study on the capability of a pair of groups. The
theory of capability of groups may be extended to the theory of
pairs of groups. In fact capable pairs are defined in terms of
J.-L. Loday's notion $[8]$ of a relative central extensions. By a
pair of groups we mean a group $G$ and a normal subgroup $N$,
denoted by $(G,N)$. If $M$ is another group on which an action of
$G$ is given, the $G$-center of $M$ is defined to be the subgroup
$$Z(M,G)= \{m\in M|m^g=m, \forall g\in G\}.$$

A relative central extension of the pair $(G,N)$ consists of a
group homomorphism $\sigma:M\rightarrow G$ together with an
action of $G$ on $M$ such that \\
i) $\sigma(M)=N$,\\
ii) $\sigma(m^g)=g^{-1}\sigma(m)g$, for all $g\in G, m\in M$,\\
iii) $m'^{\sigma(m)}=m^{-1}m'm$, for all $m,m'\in M$,\\
iv) $Ker(\sigma)\subseteq Z(M,G)$.

Now a pair of groups $(G,N)$ is said to be capable if it admits
such a relative central extension with $Ker(\sigma)=Z(M,G)$.

In section $4$, we prove that the capability of the pair
$(G_1\times G_2, N_1\times N_2)$ is equivalent to the capability
of both pairs $(G_1, N_1)$ and $(G_2, N_2)$ in some conditions.

\section{Definitions and Preliminaries}

The central subgroup $Z^*(G)$ of $G$ is defined as follows $[2]$:
$$Z^*(G)=\{\phi Z(E)\ |\ (E,\phi)\ is\ a\ central\ extension\ of\ G\}.$$
It is clear that $Z^*(G)$ is a characteristic subgroup of $G$
contained in $Z(G)$.

\begin{thm}$[2]$ (i) A group $G$ is capable if and only if
$Z^*(G)=1$.\\
(ii) Let $N$ be a central subgroup of $G$. Then $N\subseteq
Z^*(G)$ if and only if the natural map $M(G)\rightarrow M(G/N)$ is
monomorphic.\\
(iii) $Z^*(\prod_{i\in I}G_i)\subseteq \prod_{i\in I}Z^*(G_i)$,
and hence if $G_i$'s are capable groups, then $G=\prod_{i\in
I}G_i$ is also capable. \end{thm}

In general the above inclusion is proper. The following sufficient
condition forcing equality.

\begin{thm}$[2]$ Let $G=\prod_{i\in I}G_i$. Assume that for $i\neq j$ the maps
$v_i\otimes 1:Z^*(G_i)\otimes G_j/G'_j\rightarrow G_i/G'_i\otimes
G_j/G'_j$ are zero, where $v_i$ is the natural map
$Z^*(G_i)\rightarrow G_i\rightarrow G_i/G'_i$. Then
$Z^*(G)=\prod_{i\in I}Z^*(G_i)$.
\end{thm}

As a consequence of the above theorem, if $\{G_i\ |\ i\in I\}$ is
a family of finite capable groups with
$(|A_i^{ab}|,|B_j^{ab}|)=1$, for all $i\neq j$, then
$G=\prod_{i\in I}G_i$ is capable if and only if any $G_i$ is
capable.

Let $\mathcal{V}$ be a variety of groups defined by the set of
laws $V$. A group $G$ is said to be $\mathcal{V}-$capable if
there exists a group $E$ such that $G\cong E/V^*(E)$. If
$\psi:E\rightarrow G$ is a surjective homomorphism with $ker
\psi\subseteq V^*(E)$, then the intersection of all subgroups of
the form $\psi(V^*(E))$ is denoted by $(V^*)^*(G)$. It is obvious
that $(V^*)^*(G)$ is a characteristic subgroup of $G$ contained
in $V^*(G)$. If $\mathcal{V}$ is the variety of abelian groups,
then the subgroup $(V^*)^*(G)$ is the same as $Z^*(G)$ and in
this case $\mathcal{V}-$capablity is equal to capablity $[11]$.

\begin{thm} $[11]$ (i) A group $G$ is $\mathcal{V}-$capable if and only if
$(V^*)^*(G)=1$.\\
(ii) $(V^*)^*(\prod_{i\in I}^{}G_i)\subseteq \prod_{i\in
I}(V^*)^*(G_i).$
\end{thm}

As a consequence, if the $G_i$'s are $\mathcal{V}-$capable groups,
then $G=\prod_{i\in I}^{}G_i$ is also $\mathcal{V}-$capable.

Note that, in the above theorem, the equality does not hold in
general (see Example $3.5$).

\begin{thm} $[11]$ Let $N$ be a normal subgroup contained in the marginal subgroup
of $G$, $V^*(G)$. Then $N\subseteq (V^*)^*(G)$ if and only if the
homomorphism induced by the natural map
${\mathcal{V}}M(G)\rightarrow {\mathcal{V}}M(G/N)$ is a
monomorphism.
\end{thm}

\section{Capability of a Direct Product of Groups}

In this section we verify the equation $(V^*)^*(A\times
B)=(V^*)^*(A)\times (V^*)^*(B)$ for some famous varieties.

First, we note that in general, for an arbitrary variety of groups
$\mathcal{V}$, and groups $A$ and $B$, ${\mathcal{V}}M(A\times
B)\cong {\mathcal{V}}M(A)\times {\mathcal{V}}M(B)\times T$, where
$T$ is an abelian group $[10]$. But for some particular varieties,
the group $T$ is trivial with some conditions. For instance, some
famous varieties as variety of abelian groups $[10]$, variety of
nilpotent groups $[4]$, and some varieties of polynilpotent groups
$[7]$ have the property that: {\it for any two groups $A$ and $B$
with $(|A^{ab}|,|B^{ab}|)=1$ the isomorphism
${\mathcal{V}}M(A\times B)\cong {\mathcal{V}}M(A)\times
{\mathcal{V}}M(B)$ $(*)$ holds.}

Now, Suppose that ${\mathcal{V}}$ is a variety, $A$ and $B$ are
two groups with the property $${\mathcal{V}}M(A\times B)\cong
{\mathcal{V}}M(A)\times {\mathcal{V}}M(B).$$ By Theorem $2.4$, we
have the following monomorphism
$${\mathcal{V}}M(A)\times {\mathcal{V}}M(B)\hookrightarrow{\mathcal{V}}M(\frac{A}{(V^*)^*(A)})\times
{\mathcal{V}}M(\frac{B}{(V^*)^*(B)}).$$
 Moreover, we have the following inclusion
$${\mathcal{V}}M(\frac{A}{(V^*)^*(A)})\times
{\mathcal{V}}M(\frac{B}{(V^*)^*(B)})\hookrightarrow
{\mathcal{V}}M(\frac{A}{(V^*)^*(A)}\times \frac{B}{(V^*)^*(B)})$$
Finally, we get the monomorphism
$${\mathcal{V}}M(A\times
B)\hookrightarrow{\mathcal{V}}M(\frac{A\times B}{(V^*)^*(A)\times
(V^*)^*(B)}).$$ Thus, by Theorem $2.4$, we conclude that
$$(V^*)^*(A)\times (V^*)^*(B)\leq (V^*)^*(A\times B).$$ This note
leads us to
 our main result.

\begin{thm} Let $\mathcal{V}$ be a variety, $A$ and $B$ be two
groups with ${\mathcal{V}}M({A\times B})\cong
{\mathcal{V}}M({A})\times{\mathcal{V}}M({B})$, then
$(V^*)^*(A\times B)=(V^*)^*(A)\times (V^*)^*(B)$. Consequently
$A\times B$ is $\mathcal{V}$-capable if and only if $A$ and $B$
are both $\mathcal{V}$-capable.
\end{thm}

\begin{rem}
We recall that above property holds for some famous varieties as
variety of abelian groups and variety of nilpotent groups, where
$(|A_{ab}|,|B_{ab}|)=1$ $([4, 12])$. Thus by theorem $3.1$ for a
family of groups $\{A_i\ |\ 1\leq i\leq n\}$ whose
abelianizations have mutually coprime orders, $\prod_{i=1}^n A_i$
is capable ($\mathcal{N}_c$-capable) if and only if every $A_i$
is capable ($\mathcal{N}_c$-capable). Also, we note that
$\mathcal{N}_c$-capability of a group implies its
$\mathcal{N}_{c-1}$-capability, for any $c\geq 2$.
\end{rem}

Following, we have similar conclusion for variety of polynilpotent
groups in some senses.

\begin{cor} Let $\{A_i\ |\ 1\leq i\leq n\}$ be a family of
groups whose abelianizations have mutually coprime orders.  If
$\prod_{i=1}^n A_i$ is nilpotent of class at most $c_1$, then it
is $\mathcal{N}_{c_1,\cdots,c_s}$-capable if and only if every
$A_i$ is $\mathcal{N}_{c_1,\cdots,c_s}$-capable.
\end{cor}
\begin{proof}
 First, for the
variety of polynilpotent groups with above hypothesis, there
exists the following isomorphism $[7, Lemma\ 3.9]$
$$\mathcal{N}_{c_1,\cdots,c_s}(\prod_{i=1}^n A_i)\cong
(\mathcal{N}_{c_{s}}M(\cdots\mathcal{N}_{c_{2}}M(\mathcal{N}_{c_{1}}M(\prod_{i=1}^n
A_i)))\cdots).$$ Also we note to the property of variety of
nilpotent groups that $\mathcal{N}_cM(\prod_{i=1}^n A_i)\cong
\prod_{i=1}^n \mathcal{N}_cM(A_i)$, for above family of groups 
$[4, Proposiyion\ 3]$. Hence by Theorem $3.1$ the result holds. 
\end{proof}

\begin{exm} If $\{A_i|1\leq i\leq n\}$ is a family of perfect groups, then
$\prod_{i=1}^n A_i$ is $\mathcal{V}$-capable if and only if each
$A_i$ is $\mathcal{V}$-capable, where $\mathcal{V}$ may be each
of these three varieties, variety of abelian groups, variety of
nilpotent groups, or variety of polynilpotent groups.
\end{exm}

Note that Ellis gave a similar result for $\mathcal{N}_c$-capability of a
direct product of groups $[4,\ Theorem 2]$ with another method.

\begin{exm}
i) Let $G\cong
\mathbb{Z}_{n_1}\oplus\mathbb{Z}_{n_2}\oplus\cdots\oplus\mathbb{Z}_{n_k}$,
where $k\geq 3$, $n_{i+1}|n_{i}$ for all $1\leq i\leq k-1$ and
$n_1=n_2=n_3$. Then by $[9, lamma\ 3.4]$ $G$ is
$\mathcal{N}_{c_1,\cdots,c_t}$-capable if $t\geq 2$ and $c_1=1$; but by $[9, lamma\ 3.3]$ no one of its direct summands, $\mathbb{Z}_{n_i}$, $1\leq i\leq k$, is $\mathcal{N}_{c_1,\cdots,c_t}$-capable. This shows that we can not omit the condition of being mutually coprime orders for abelianizations of the family of groups $\{A_i\ |\ 1\leq i\leq n\}$ in corollary $3.3$.\\
ii) Let $G\cong
\mathbb{Z}_{n_1}\oplus\mathbb{Z}_{n_2}\oplus\cdots\oplus\mathbb{Z}_{n_k}$,
where $k\geq 2$, $n_{i+1}|n_{i}$ for all $1\leq i\leq k-1$ and
$n_1=n_2$. Then by $[9, lamma\ 3.7]$ $G$ is
$\mathcal{N}_{c_1,\cdots,c_t}$-capable if $t=1$ or $c_1\geq 2$; but no one of its direct summands, $\mathbb{Z}_{n_i}$, $1\leq i\leq k$, is $\mathcal{N}_{c_1,\cdots,c_t}$-capable. This example also shows that one can not omit the condition of being mutually coprime orders for abelianizations of the family of groups $\{A_i\ |\ 1\leq i\leq n\}$ in corollary $3.3$ for the variety of nilpotent groups, $\mathcal{N}_{c}$ or the variety of polynilpotent groups $\mathcal{N}_{c_1,\cdots,c_t}$ where $c_1\geq 2$.\\
iii) Put $A\cong \mathbb{Z}_{n}\oplus\mathbb{Z}_{n}\cong B$. Then $(A^{ab},B^{ab})\neq 1$ and $A\oplus B$ is $\mathcal{N}_{c_1,\cdots,c_t}$-capable where $t=1$ or $c_1\geq 2$. Also its direct summands, $A$ and $B$, are $\mathcal{N}_{c_1,\cdots,c_t}$-capable. This shows that the condition of being mutually coprime orders for abelianizations of the the family $\{A_i\ |\ 1\leq i\leq n\}$ is not necessary condition for transferring the varietal capability of a direct product to its factors.
\end{exm}

\section{Capability of a Direct Product of Pair of Groups}

In order to study the capability of a  pair of groups $(G,N)$,
Ellis $[3]$ introduced a subgroup $Z^{\wedge}_{G}(N)$ with the
property that the pair is capable if and only if
$Z^{\wedge}_{G}(N)=1$. To define this subgroup, we need  to
recall the definition of exterior product from $[3]$ as follows.

\begin{defn} Let $N$ and $P$ be arbitrary subgroups of $G$. The
exterior product $P\wedge N$ is the group generated by symbols
$p\wedge n$, for $p\in P$, $n\in N$, subject to the relations\\
1) $pp'\wedge n=(p'^{p}\wedge n^p)(p\wedge n)$,\\
2) $p\wedge nn'=(p\wedge n)(p^n\wedge n'^n)$,\\
3) $x\wedge x=1$, \\
for $x\in P\cap N$, $n, n'\in N$, $p, p'\in P$.
\end{defn}

For a group $G$ and normal subgroups $N$ and $P$, the exterior
P-center of $N$ is denoted by $Z^{\wedge}_{P}(N)$, and is defined
to be
$$\{n\in N|1=p\wedge n\in P\wedge N,\ for\ all p\in P \}.$$
Ellis $[3]$ proved that the pair $(G,N)$ is capable if and only if
$Z^{\wedge}_{G}(N)=1$. In $[12,\ Corollary\ 5.3]$ it is proved
that if $(G,N)$ is a pair of abelian groups, then
$Z^{\wedge}_{G}(N)=Z^*(G)\cap N$. By a similar proof, we have the
result without the condition of abelianess. Now, we can give the
following result about the capability of the direct product of
pair of groups.

\begin{thm} (it??) Suppose that $(G_1, N_1)$ and $(G_2, N_2)$ are two pair of groups with
$(|(G_{1})^{ab}|,|(G_{2})^{ab}|)=1$, then
$$Z^{\wedge}_{{G_1\times G_2}}(N_1\times
N_2)=Z^{\wedge}_{{G_1}}(N_1)\times Z^{\wedge}_{{G_2}}(N_2).$$
Consequently the capability of $(G_1\times G_2, N_1\times N_2)$ is
equivalent to the capability of both pairs $(G_1, N_1)$ and
$(G_2, N_2)$.
\end{thm}
\begin{proof}
Since $(|(G_{1})^{ab}|,|(G_{2})^{ab}|)=1$, we have $M({G_1\times
G_2})\cong
 M({G_1})\times M({G_2})$ and by Theorem $3.1$, $Z^*({G_1\times
G_2})=Z^*({G_1})\times Z^*({G_2})$. Finally using the ??previous note, we conclude that
$$Z^{\wedge}_{{G_1\times G_2}}(N_1\times N_2)=Z^*({G_1\times G_2})\cap
(N_1\times N_2)$$
$$=Z^*({G_1})\times Z^*({G_2})\cap (N_1\times N_2)$$
$$=(Z^*({G_1})\cap N_1)\times (Z^*({G_2})\cap N_2)$$
$$=Z^{\wedge}_{{G_1}}(N_1)\times
Z^{\wedge}_{{G_2}}(N_2).$$
This equation and $[5, Theorem\ 3]$ complete the proof.
\end{proof}

\begin{rem}By induction we can also conclude the above
theorem for a family
of pairs of groups $\{(G_i,N_i)\ |\ 1\leq i\leq n\}$, where
$(G_i)^{ab}$'s have mutually coprime orders, that is, $(\prod_{i=1}^n
G_i,\prod_{i=1}^n N_i)$ is a capable pair of groups if and only if every
$(G_i,N_i)$ is a capable pair of groups.\end{rem}

\begin{exm}
Using $[12, \ Theorem\ 5.4]$ $(\mathbb{Z}_6\times
\mathbb{Z}_6,\mathbb{Z}_3\times \mathbb{Z}_3)$ is a capable pair
of groups, but $(\mathbb{Z}_6, \mathbb{Z}_3)$ is not. This shows
that the condition $(|(G_1)^{ab}|,|(G_2)^{ab}|)=1$ in Theorem $4.2$ can not be omitted.
\end{exm}

\end{document}